\documentclass[11pt,a4paper]{article}

\usepackage[english]{babel}
\usepackage[utf8]{inputenc}
\usepackage[T1]{fontenc}
\usepackage{lmodern} \normalfont 
\DeclareFontShape{T1}{lmr}{bx}{sc} { <-> ssub * cmr/bx/sc }{}
\usepackage{parskip}
\usepackage{microtype}

\usepackage{amssymb}
\usepackage[intlimits]{amsmath}
\usepackage{amsthm}
\usepackage{mathtools}
\usepackage{mathrsfs}
\usepackage{etoolbox}
\usepackage{physics}
\usepackage{siunitx}
\begingroup
    \makeatletter
    \@for\theoremstyle:=definition,remark,plain\do{%
        \expandafter\g@addto@macro\csname th@\theoremstyle\endcsname{%
            \addtolength\thm@preskip\parskip
            }%
        }
\endgroup

\usepackage{algorithm,algorithmicx,algpseudocode}

\usepackage{tikz}
\usetikzlibrary{calc}
\usepackage{pgfplots}
\pgfplotsset{compat=newest}
\usepackage{color}
\usepackage{graphicx}

\usepackage[left=3cm,right=3cm,top=3cm,bottom=3cm]{geometry}
\usepackage{booktabs}
\usepackage{subcaption}

\usepackage{todonotes}
\usepackage{cite}

\usepackage[acronym,nomain,nogroupskip,toc,nonumberlist,nopostdot,style=super]{glossaries}

\makeglossaries

\newacronymstyle{emphAcronym}
{%
  \GlsUseAcrEntryDispStyle{long-short}%
}%
{%
  \GlsUseAcrStyleDefs{long-short}%
}
\setacronymstyle{emphAcronym}

\let\oldbibliography\thebibliography
\renewcommand{\thebibliography}[1]{%
  \small
  \oldbibliography{#1}%
  \setlength{\itemsep}{0pt}%
}

\usepackage[colorlinks,citecolor=black]{hyperref}
\usepackage[nameinlink]{cleveref}

\numberwithin{equation}{section}
\numberwithin{table}{section}
\numberwithin{figure}{section}

\newtheorem{theorem}{Theorem}[section]
\newtheorem{problem}[theorem]{Problem}

\newtheorem{assumption}[theorem]{Assumption}

\crefname{lemma}{lemma}{lemmata}
\Crefname{lemma}{Lemma}{Lemmata}

\crefname{corollary}{corollary}{corollaries}
\Crefname{corollary}{Corollary}{Corollaries}

\theoremstyle{definition}
\newtheorem{remark}[theorem]{Remark}
\AtEndEnvironment{remark}{\hfill\ensuremath{\diamondsuit}}
\theoremstyle{definition}

\newtheorem{example}[theorem]{Example}
\AtEndEnvironment{example}{\hfill\ensuremath{\bigcirc}}

\DeclarePairedDelimiter{\ceil}{\lceil}{\rceil}
\DeclarePairedDelimiter{\floor}{\lfloor}{\rfloor}

\newcommand{\bodePlotWidth}{5in}
\newcommand{\bodePlotHeight}{2in}
\newcommand{\original}{original}

\newcommand{\structRealization}{struct. ROM}
\newcommand{\structRealizationParam}{struct. ROM with $\delay^\star$}
\newcommand{\interpDataText}{interpolation data}
\newcommand{\interpDataParam}{test data}
\newcommand{\absError}{abs. error}
\newcommand{\absErrorParam}{abs. error with $\delay^\star$}

\newcommand{\mat}[3]{#1^{#2\times #3}}
\newcommand{\pseudo}[1]{#1^{\dagger}}

\newcommand{\bs}[1]{\boldsymbol{#1}}
\newcommand{\state}{\bs{x}}
\newcommand{\stateDim}{n}
\newcommand{\dimFOM}{N}
\newcommand{\inpVar}{\bs{u}}
\newcommand{\inpVarDim}{m}
\newcommand{\outVar}{\bs{y}}
\newcommand{\outVarDim}{\ell}

\newcommand{\param}{\bs{p}}

\newcommand{\paramSet}{\mathbb{P}}
\newcommand{\paramSetDim}{p}
\newcommand{\delay}{\tau}

\newcommand{\system}{\Sigma}
\newcommand{\A}{A}
\newcommand{\B}{B}

\newcommand{\C}{C}

\newcommand{\E}{E}
\newcommand{\Aa}{A_1}
\newcommand{\Ab}{A_2}

\newcommand{\red}[1]{\widetilde{#1}}

\newcommand{\outVarRed}{\red{\outVar}}

\newcommand{\Ared}{\red{\A}}
\newcommand{\Bred}{\red{\B}}

\newcommand{\Cred}{\red{\C}}

\newcommand{\hfunc}{\mathfrak{h}}

\newcommand{\transfer}{H}
\newcommand{\transferRed}{\red{\transfer}}
\newcommand{\frequency}{s}

\def\numData{n}

\def\leftPoint{\mu}

\def\leftData{f}
\def\LeftData{\mathcal{F}}

\def\rightPoint{\sigma}

\def\rightData{g}
\def\RightData{\mathcal{G}}

\newcommand{\pole}{\lambda}
\newcommand{\interpData}{\vartheta}

\newcommand{\numTestData}{q}
\newcommand{\lsError}{\mathcal{E}}

\def\numFunctions{K}
\def\Hred{\smash{\widetilde{H}}}

\def\Bred{\smash{\widetilde{\B}}}

\def\Cred{\smash{\widetilde{\C}}}

\def\Kred{\smash{\widetilde{\mathcal{K}}}}

\newcommand{\timeStep}{\delta_t}
\newcommand{\numTimeSteps}{N}
\newcommand{\ZinpVar}{\widehat{\inpVar}}
\newcommand{\ZoutVar}{\widehat{\outVar}}
\newcommand{\outVarMat}{\bs{Y}}

\newcommand{\tf}{t_\mathrm{f}}
\newcommand{\solOperator}{\mathcal{S}}
\newcommand{\romOperator}{\red{\solOperator}}
\newcommand{\tol}{\varepsilon}
\newcommand{\Linfty}{\mathcal{L}_\infty}
\newcommand{\Ltwo}{\mathcal{L}_2}
\newcommand{\Htwo}{\mathcal{H}_2}
\newcommand{\LinftyNorm}[1]{\left\|#1\right\|_{\Linfty}}
\newcommand{\LtwoNorm}[1]{\left\|#1\right\|_{\Ltwo}}
\newcommand{\HtwoNorm}[1]{\left\|#1\right\|_{\Htwo}}
\newcommand{\indSet}{\mathcal{I}}
\newcommand{\markovParam}{h}
\newcommand{\fourierPoint}[1]{q_{#1}}
\newcommand{\sphere}{\mathbb{S}}
\newcommand{\transferEstimate}{\hat{\transfer}}
\newcommand{\FourierMatrix}{\bs{F}}
\newcommand{\jmin}{j_{\mathrm{min}}}

\newcommand{\fmin}{f_{\mathrm{min}}}
\newcommand{\fmax}{f_{\mathrm{max}}}

\DeclareMathOperator*{\argmin}{arg\,min}

\newacronym{MOR}{MOR}{model order reduction}
\newacronym{ROM}{ROM}{reduced-order model}
\newacronym{FOM}{FOM}{full order model}
\newacronym{ODE}{ODE}{ordinary differential equation}
\newacronym{DDE}{DDE}{delay differential equation}
\newacronym{DDAE}{DDAE}{delay differential-algebraic equation}
\newacronym{SDE}{SDE}{stochastic differential equation}
\newacronym{SISO}{SISO}{single-input/single-output}
\newacronym{SDDE}{SDDE}{stochastic delay differential equation}
\newacronym{ETFE}{ETFE}{empirical transfer function estimate}
\newacronym{DFT}{DFT}{discrete Fourier transform}
\newacronym{FFT}{FFT}{fast Fourier transform}
\newacronym{LTI}{LTI}{linear time-invariant}
\newacronym{DMD}{DMD}{dynamic mode decomposition}
\newacronym{PDE}{PDE}{partial differential equation}
\newacronym{SVD}{SVD}{singular value decomposition}
\newacronym{BIBO}{BIBO}{bounded-input/bounded-output}
\newacronym{lsETFE}{lsTFE}{least-squares transfer function estimate}

\title{From Time-Domain Data to Low-Dimensional Structured Models}
\author{Elliot Fosong\footnotemark[1] \and Philipp Schulze\footnotemark[2]~\footnotemark[3] \and Benjamin Unger\footnotemark[2]~\footnotemark[4]}

\begin{document}

\maketitle
\renewcommand{\thefootnote}{\fnsymbol{footnote}}
\footnotetext[1]{Pembroke College,
University of Cambridge, 
Cambridge, CB2 1RF, UK, \texttt{einf2@cam.ac.uk}.
The work of this author was supported by the German Academic Exchange Service within the RISE Germany program.}
\footnotetext[2]{Institut f\"ur Mathematik,
Technische Universität Berlin, Str.\ des 17.~Juni~136,
10623~Berlin,
Germany,
\texttt{\{pschulze,unger\}@math.tu-berlin.de}. }
\footnotetext[3]{
The work of this author is supported by the DFG Collaborative Research Center 1029 \emph{Substantial efficiency increase in gas turbines through direct use of coupled unsteady combustion and flow dynamics}, project A02.}
\footnotetext[4]{
The work of this author is supported by the DFG Collaborative Research Center 910 \emph{Control of self-organizing nonlinear systems: Theoretical methods and concepts of application}, project A2.}
%

\begin{abstract}
	We present a framework for constructing a structured realization of a linear time-invariant dynamical system solely from a discrete sampling of an input and output trajectory of the system. We estimate the transfer function of the original model at selected frequencies using a modification of the empirical transfer function estimation that was recently presented in [Peherstorfer, Gugercin, Willcox, SIAM J.~Sci.~Comput., 39(5):2152--2178, 2017]. Our realization interpolates the transfer function estimates and can be seen as a generalization of the Loewner framework to structured systems. We demonstrate the presented framework by means of a delay example.
\end{abstract}
\noindent
{\bf Keywords:} structured realization; nonintrusive model reduction; structure-preserving model reduction; delay differential equations; transfer function estimate
\vskip .3truecm
\noindent
{\bf AMS(MOS) subject classification:} 30E10, 37M99, 65P99, 93C05

\section{Introduction}
\label{sec:introduction}
Nowadays it is common to describe a physical or chemical system by a mathematical surrogate model. The demand for high fidelity models results in large-scale dynamical systems, for which classical numerical methods may be too time or memory consuming. In these cases, \gls{MOR} aims in reducing the computational cost by approximating the dynamical system with a model of a smaller dimension, the so-called \gls{ROM}. For an overview of existing methods we refer to the recent books and surveys \cite{BenCOW17,BauBF14,QuaMN16,Ant05,benner2015survey,hesthaven2016certified,chiHRW16}. Most \gls{MOR} schemes are intrusive in the sense that they are formulated in a projection framework and thus require access to an internal state-space representation of the complex physical system. 

In practical applications the model might be available implicitly via a simulation code and even depend on look-up tables. Thus, the model itself can be considered as a black box providing simulation results or measurements for given inputs (cf.\ \Cref{fig:systemPicture}) without the possibility to access a state-space realization. In this case the surrogate model needs to be built from input/output measurements only. Popular methods that are tailored to such a data-driven setting are the Loewner framework \cite{MayA07}, vector fitting \cite{DrmGB15a,GusS99}, or \gls{DMD} \cite{Sch10,TuRLBK14,KutBBP16}. 

\begin{figure}[ht]
	\centering
	\begin{tikzpicture}
		\node (rect) at (1,0) [draw,very thick,minimum width=2cm,minimum height=1cm] {$\system$};
		\draw[very thick,->] (-1.5,0) -- node[above] {$\inpVar$} (0,0);
		\draw[very thick,->] (2,0) -- node[above] {$\outVar$} (3.5,0);
	\end{tikzpicture}
	\caption{The system $\system$ is considered a black-box.}
	\label{fig:systemPicture}
\end{figure}
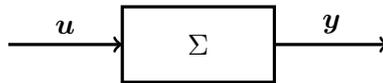

Despite the inaccessibility of a state-space description, there may yet be a good understanding of the behavior of the system. Possible examples are vibration effects that are naturally associated with second-order systems and advection effects, which are associated with state delays, cf{.} \cite{SchUBG18}. 
Further examples of relevant system structures are listed in \Cref{tab:StructureExamples}.

\begin{table}
	\centering
	\caption{Examples of system structures and their transfer functions \cite[Table~1]{SchUBG18}}
	\label{tab:StructureExamples}
	{\footnotesize
	\begin{tabular}{lll}
		\toprule
		& \textbf{state space description} & \textbf{transfer function}\\\midrule
		second-order & $A_1\ddot{\state}(t) + A_2\dot{\state}(t) + A_3\state(t) = B\inpVar(t)$ & $C\left(s^2A_1 + sA_2 + A_3\right)^{-1}B$\\
		state delay & $A_1\dot{\state}(t) + A_2\state(t) + A_3\state(t-\delay) = B\inpVar(t)$ & $C\left(sA_1 + A_2 + \mathrm{e}^{-\delay s}A_3\right)^{-1}B$\\
		neutral delay & $A_1\dot{\state}(t) + A_2\state + A_3\dot{\state}(t-\delay) = B\inpVar(t)$ & $C\left(sA_1 + A_2 + s\mathrm{e}^{-\tau s}A_3\right)^{-1}B$\\
		viscoelastic & $A_1\ddot{\state}(t) + \int_0^t h(t\!-\!\tau)A_2\dot{\state}(\tau)\mathrm{d}\tau + A_3\state(t) = B\inpVar(t)$ & $C\left(s^2A_1 + s\hat{h}(s)A_2 + A_3\right)^{-1}\!B$\\\bottomrule
	\end{tabular}
	}
\end{table}

In these cases it is desirable to reflect such structural properties within the realization, since structure preservation may result in \glspl{ROM} of smaller dimension than what unstructured methods produce while maintaining a comparable or even better accuracy, see \cite[Section 5]{BeaG09}. Although there is a significant body of literature dealing with structure-preserving \gls{MOR} methods \cite{SuC91,MeyS96,BeaG09,ChaBG16,Fre08,LalKM03,ChaGVV05}, almost all of the approaches require an internal description. Notable exceptions are provided in \cite{ScaA14,PonPS15,SchU16,SchUBG18}.
 
 \begin{remark}
 	For some model problems, for instance a circuit that involves a lossless transmission line \cite{Bra67}, it is possible to transform a hyperbolic \gls{PDE} into a delay equation \cite{Lop76,CooK68} that is -- from a computational perspective -- much easier to solve. Thus even if a state-space description is available, it may be advantageous to choose a different structure for the surrogate model than for the original model. Notice that many of these problems are characterized by slowly decaying Hankel singular values or Kolmogorov $n$-widths \cite{UngG18}, which prevents classical \gls{MOR} methods from succeeding and thus requires a special treatment \cite{OhlR16,ReiSSM18,CagMS19}. 
 \end{remark}

Throughout this paper we make the assumption that the system $\system$ in \Cref{fig:systemPicture} is linear, i.e., there exists a \gls{LTI} operator $\solOperator$ with $\outVar = \solOperator\inpVar$. For simplicity, we additionally assume a \gls{SISO} system, that is $\inpVar(t), \outVar(t)\in\mathbb{R}$ for all $t$. According to our discussion above, we are interested in solving the following problem.

\begin{problem}
	\label{problem:generalSetting}
	Construct a structured \gls{LTI} operator $\romOperator$ solely from input/output data such that
	\begin{equation*}
		\|\outVar - \outVarRed\| = \|\solOperator\inpVar - \romOperator\inpVar\| \leq \tol \|\inpVar\|
	\end{equation*}
	for all admissible input signals $\inpVar$, a small parameter $\tol\geq 0$, and suitable norms.
\end{problem}

In order to solve \Cref{problem:generalSetting} we have to address the question what kind of data we assume available and define precisely, what a structured \gls{LTI} operator is. Recall that \gls{LTI} systems can be represented in the time domain or in the frequency domain \cite{Ant05}. The mapping from the time to the frequency domain is given by the Laplace transform for continuous-time systems and the Z-transform for discrete-time systems. Moreover, the $\Linfty$ error in the time domain can be bounded by the $\Htwo$ error in the frequency domain via
\begin{equation}
	\label{eq:timeFrequencyErrorBound}
	\LinftyNorm{\outVar-\outVarRed} \vcentcolon= \sup_{t>0} \|\outVar(t)-\outVarRed(t)\|_\infty \leq \HtwoNorm{\solOperator - \smash{\romOperator}}\LtwoNorm{\inpVar}.
\end{equation}
In fact, for \gls{SISO} systems, the $\Htwo$ norm is the $\Ltwo$-$\Linfty$ induced norm of the underlying convolution operator, i.e. $\HtwoNorm{\solOperator - \smash{\romOperator}}$ is the smallest number such that \eqref{eq:timeFrequencyErrorBound} holds for all inputs $\inpVar\in\Ltwo$ \cite{BeaG17}.

It is well-known (cf.\ \cite{BeaG17} and the references therein) that if the solution operator $\solOperator$ is the convolution operator of a standard state-space realization, that is (assuming a zero initial condition and no direct feed through),
\begin{displaymath}
	(\solOperator\inpVar)(t) = \int_0^t \C\exp(\A(t-s))\B\inpVar(s)\mathrm{d}s,
\end{displaymath}
the $\Htwo$ error $\smash{\HtwoNorm{\solOperator - \smash{\romOperator}}}$ is minimized if the transfer function of $\romOperator$ interpolates the transfer function of $\solOperator$ at the mirror images of the poles of $\romOperator$. Thus our approach to solve \Cref{problem:generalSetting} is to construct $\romOperator$ such that it is an interpolant of $\solOperator$ in the frequency domain. Since we only assume access to input/output measurements in the time domain, the transfer function of $\solOperator$ is not available and hence needs to be estimated (see the forthcoming \Cref{sec:lsTFE}).

\begin{remark}
	For structured systems, interpolation in the frequency domain was investigated in \cite{BeaG09} within an projection framework and in \cite{SchUBG18} in a data-driven framework. Although it is possible to construct $\Htwo$-optimal interpolants solely from data \cite{BeaG12}, the optimality conditions for structured problems are much more involved \cite{BeaB14,PonGBPS16} and to our knowledge, there exists no general computational strategy to obtain optimal interpolation points.
\end{remark}

The remainder of the paper is structured as follows: In \Cref{sec:lsTFE} we review the transfer function estimation algorithm presented in \cite{PehGW17} and adapt the theoretical analysis to our setting (cf. \Cref{thm:convergenceETFE}). The connection between the continuous-time and the discrete-time setting is given in \Cref{sec:implementationDetailslsETFE}. To make the paper self-contained, we review the structured interpolation framework from \cite{SchUBG18} in \Cref{sec:structuredRealization} and present a simple algorithm for estimating parameters which are associated to the specific structure -- for instance, a time delay -- in \Cref{sec:parameterEstimation}. We illustrate the theoretical findings with a numerical case study for a delay example in \Cref{sec:caseStudy}.

\section{Least-Squares Transfer Function Estimate}
\label{sec:lsTFE}
Several methods, such as for instance the so-called \emph{signal generator approach} \cite{Ast10}, are designed to use time-domain data to obtain frequency measurements. In this work, we use a modification of the \gls{ETFE} method \cite{Lju85} that is presented in \cite{PehGW17} and does not assume periodicity of the input and output sequence. Since the main tool of the method from \cite{PehGW17} is the solution of a least-squares problem (see \eqref{eq:leastSquaresETFE}) we refer to this approach as \gls{lsETFE}.

For a given time step size $\timeStep>0$ consider the time grid $0 = t_0 < t_1 < \ldots < t_\numTimeSteps = \tf$ with $t_j = j\timeStep$ for $j=0,1,\ldots,\numTimeSteps$ and $\numTimeSteps\in\mathbb{N}$. Moreover, we assume that measurements of the input and the output at the time grid are available, i.e., that we have access to the data
\begin{equation}
	\label{eq:timeDomainData}
	\inpVar_j \vcentcolon= \inpVar(t_j)\in\mathbb{R}^{\inpVarDim},\qquad\text{and}\qquad
	\outVar_j \vcentcolon= \outVar(t_j)\in\mathbb{R}^{\outVarDim}\qquad\text{for}\ j=0,1,\ldots,\numTimeSteps.
\end{equation}
For simplicity, we assume in the following, that the data under consideration is generated from a \gls{SISO} dynamical system, that is, $\inpVarDim = \outVarDim = 1$, with zero initial condition. Moreover, we assume that the system is \gls{BIBO} stable, i.e., the sequence $(\outVar_j)_{j\in\mathbb{N}}$ is bounded for any bounded sequence $(\inpVar_j)_{j\in\mathbb{N}}$, and make the following crucial assumption for the remainder of this chapter.

\begin{assumption}
	\label{ass:LTIsystem}
	The data in \eqref{eq:timeDomainData} is generated from a causal, \gls{BIBO} stable \gls{LTI} system.
\end{assumption}

For the case $\numTimeSteps = \infty$, \Cref{ass:LTIsystem} guarantees that the output data $\outVar_j$ is obtained via the convolution of the impulse response of the system and the inputs $\inpVar_j$. More precisely, there exist numbers $\markovParam_i\in\mathbb{R}$ such that
\begin{equation*}
	\label{eq:convolution}
	\outVar_j = \sum_{i=0}^j \markovParam_i \inpVar_{j-i}\qquad\text{for}\ j\in\mathbb{N}.
\end{equation*}
\begin{example}
	\label{ex:discreteStandardSystem}
	If the data \eqref{eq:timeDomainData} is generated from the discrete-time system
	\begin{align*}
		\E\state_{j+1} &= \A\state_j + \B\inpVar_j,\\
		\outVar_j &= \C\state_j, \\
		\state_0 &= 0,
	\end{align*}
	with nonsingular matrix $\E\in\mat{\mathbb{R}}{\stateDim}{\stateDim}$, then $\markovParam_i = \C\left(\E^{-1}\A\right)^{i-1}\left(\E^{-1}\B\right)$ for $i>0$ and $\markovParam_0 = 0$.
\end{example}

Taking the Z-transforms of $\left(\inpVar_j\right)_{j\in\mathbb{N}}$ and $\left(\outVar_j\right)_{j\in\mathbb{N}}$ 
\begin{equation*}
	\label{eq:ztransforms}
	\ZinpVar(z) = \sum_{i=0}^\infty \inpVar_i z^{-i}\qquad \text{and}\qquad
	\ZoutVar(z) = \sum_{i=0}^\infty \outVar_i z^{-i}
\end{equation*}
implies $\ZoutVar(z) = \transfer(z)\ZinpVar(z)$, where $\transfer$ is given by the formal power series
\begin{displaymath}
	\transfer(z) = \sum_{i=0}^\infty \markovParam_i z^{-i}.
\end{displaymath}

In practical applications we have $\numTimeSteps < \infty$ and thus cannot apply the Z-transform. Instead, we use the \gls{FFT}, which can be interpreted as a special case of the Z-transform. More precisely, we define $\fourierPoint{k} \vcentcolon= \exp(\tfrac{2\pi\imath}{\numTimeSteps}k)$,
\begin{equation*}
	\label{eq:ZtransformedData}
	\ZinpVar_{k;\numTimeSteps} \vcentcolon= \sum_{j=0}^{\numTimeSteps-1} \inpVar_j\fourierPoint{k}^{-j}, \qquad\text{and}\qquad
	\ZoutVar_{k;\numTimeSteps} \vcentcolon= \sum_{j=0}^{\numTimeSteps-1} \outVar_j\fourierPoint{k}^{-j}
\end{equation*}
for $k=0,\ldots,\numTimeSteps-1$. Using the index set
\begin{displaymath}
	\indSet \vcentcolon= \left\{k\in\{0,1,\ldots,\numTimeSteps\}\,\bigg|\, \abs{\ZinpVar_{k;\numTimeSteps}}>0\right\} =\vcentcolon \{k_1,\ldots,k_r\},
\end{displaymath}
we can define
\begin{displaymath}
	\transfer_{k;\numTimeSteps} \vcentcolon= \frac{\ZoutVar_{k;\numTimeSteps}}{\ZinpVar_{k;\numTimeSteps}}\qquad\text{for}\ k\in\indSet
\end{displaymath}
as an approximation of the transfer function. This particular way of estimating the transfer function is known as the \gls{ETFE} \cite{Lju85}. If the sequence $\inpVar_j$ and $\outVar_j$ are periodic with period $\numTimeSteps$, then $\transfer_{k;\numTimeSteps} = \transfer(\fourierPoint{k})$. 
In practical applications, periodicity cannot always be assumed and thus, we pursue a different way here. Following \cite{PehGW17} we define the partial sum
\begin{equation*}
	\label{eq:partialTransfer}
	\transfer_j(z) \vcentcolon= \sum_{i=0}^j \markovParam_i z^{-i}.
\end{equation*}
Using the inverse \gls{FFT},
we observe
\begin{equation}
	\label{eq:timeFrequencyRelation}
	\begin{aligned}
	\outVar_j &= \sum_{i=0}^j \markovParam_i \inpVar_{j-i} = \sum_{i=0}^j \markovParam_i \left(\frac{1}{\numTimeSteps}\sum_{k=0}^{\numTimeSteps-1} \ZinpVar_k \fourierPoint{k}^{j-i}\right) = \frac{1}{\numTimeSteps}\sum_{k=0}^{\numTimeSteps-1} \ZinpVar_k \transfer_j(\fourierPoint{k})\fourierPoint{k}^j\\
	&= \frac{1}{\numTimeSteps} \sum_{k\in\indSet} \ZinpVar_k\transfer_j(\fourierPoint{k})\fourierPoint{k}^j.
	\end{aligned}
\end{equation}
Note that \eqref{eq:timeFrequencyRelation} provides a direct link between the time domain data $\outVar_j$ and the frequency data $\transfer_j(\fourierPoint{k})$. In addition, we have the following convergence result, which is a generalization of \cite[Proposition~3.2]{PehGW17}.

\begin{theorem}
	\label{thm:convergenceETFE}
	Suppose the data \eqref{eq:timeDomainData} is generated from a dynamical system that satisfies \Cref{ass:LTIsystem}. Then
	\begin{displaymath}
		\lim_{j\to\infty} \transfer_j(z) = \transfer(z)\qquad\text{for all}\ z\in\sphere \vcentcolon= \{z\in\mathbb{C} \mid \abs{z} = 1\}.
	\end{displaymath}
\end{theorem}

\begin{proof}
	Let $z\in\sphere$. Then
	\begin{displaymath}
		\sum_{i=0}^j \abs{\markovParam_i z^{-i}} = \sum_{i=0}^j \abs{\markovParam_i}\abs{z}^{-i} = \sum_{i=0}^j \abs{\markovParam_i}.
	\end{displaymath}
	\Cref{ass:LTIsystem} implies that the system is \gls{BIBO} stable, which is equivalent to the absolute convergence of the power series $\sum_{i=0}^\infty \markovParam_i$ \cite[Chapter~2.8]{ZieTF98}. Thus the formal power series $\transfer(z)$ converges for every $z\in\sphere$, which completes the proof.
\end{proof}

\begin{remark}
	\label{rem:rateOfConvergence}
	If the data is generated from the linear system in \Cref{ex:discreteStandardSystem}, then the rate of convergence depends on the spectral radius of $\E^{-1}\A$, cf.\ \cite[Proposition~3.2]{PehGW17} and \cite[Theorem~5.18]{Ant05}. More precisely, let $\rho\geq 0$ denote the spectral radius of $\E^{-1}\A$, i.e., the modulus of the largest eigenvalue of $\E^{-1}\A$. Then there exists a constant $c\in\mathbb{R}$ independent of $j$ and $\rho$ such that
	\begin{displaymath}
		\abs{\transfer_j(z)-\transfer(z)} \leq c\rho^j
	\end{displaymath}
	for all $z\in\sphere$.
\end{remark}

The relation \eqref{eq:timeFrequencyRelation} together with the convergence result given by \Cref{thm:convergenceETFE} motivates to solve the least-squares problem
\begin{equation}
	\label{eq:leastSquaresETFE}
	\argmin_{\transferEstimate_{k_i;\numTimeSteps}} \sum_{j=\jmin}^{\numTimeSteps} \left(\outVar_j - \frac{1}{\numTimeSteps}\sum_{i=1}^r \ZinpVar_{k_i} \transferEstimate_{k_i;\numTimeSteps}\fourierPoint{k_i}^j\right),
\end{equation}
with some number $\jmin\in\mathbb{N}$ that is chosen in accordance with the expected rate of convergence in \Cref{thm:convergenceETFE} and \Cref{rem:rateOfConvergence}, cf{.} \cite{PehGW17}. For more details on the choice of $\jmin$, we refer to \cite[Section~3.6]{PehGW17}. The (minimum norm) solution of \eqref{eq:leastSquaresETFE} is obtained by computing the Moore--Penrose pseudo-inverse of the matrix
\begin{equation*}
	\label{eq:FourierMatrix}
	\FourierMatrix \vcentcolon= \frac{1}{\numTimeSteps}\begin{bmatrix}
		\ZinpVar_{k_1}\fourierPoint{k_1}^{\jmin} & \ldots & \ZinpVar_{k_r}\fourierPoint{k_r}^{\jmin}\\
		\vdots & \ddots & \vdots\\
		\ZinpVar_{k_1}\fourierPoint{k_1}^{\numTimeSteps} & \ldots & \ZinpVar_{k_r}\fourierPoint{k_r}^{\numTimeSteps}
	\end{bmatrix}\in\mat{\mathbb{C}}{(\numTimeSteps-\jmin+1)}{r},
\end{equation*}
which is given by $\pseudo{\FourierMatrix} = \mathcal{V}\Sigma^{-1}\mathcal{U}^*$, where $\mathcal{U}\Sigma\mathcal{V}^* = \FourierMatrix$ denotes the rank-revealing \gls{SVD} of $\FourierMatrix$. Note that inverting tiny but nonzero singular values in $\Sigma$ poses a numerical problem. Truncating small singular values during the computation of the pseudo-inverse amounts to solving the regularized least-squares problem (cf.\ \cite{BenHM18})
\begin{equation}
	\label{eq:regularizedLeastSquaresETFE}
	\argmin_{\transferEstimate\in\mathbb{C}^{r}} \|\FourierMatrix\transferEstimate - \outVarMat\|_2^2 + \beta\|\transferEstimate\|_2^2,
\end{equation}
with
\begin{displaymath}
	\transferEstimate = \begin{bmatrix}
		\transferEstimate_{k_1;\numTimeSteps} & \ldots & \transferEstimate_{k_r;\numTimeSteps}
	\end{bmatrix}^T\qquad \text{and}\qquad
	\outVarMat = \begin{bmatrix}
		\outVar_{\jmin} & \ldots & \outVar_{\numTimeSteps}
	\end{bmatrix}^T.
\end{displaymath}
Note that the matrix $\FourierMatrix$ is dense and, depending on the number $r$ of nonzero Fourier coefficients of the input signal, the numerical solution of \eqref{eq:regularizedLeastSquaresETFE} may become unmanageably expensive. If the user is free to choose the input signal $\inpVar_j$ then the numerical issues can be reduced as follows, see also \cite{PehGW17}: It is likely that the numerical rank deficiency of $\FourierMatrix$ is avoided, if $r$ is small, i.e., if only a small number of the Fourier coefficients of the input sequence $\inpVar_j$ is nonzero and $\numTimeSteps-\jmin$ is large enough. In particular, this ensures that the least-squares problem \eqref{eq:leastSquaresETFE} is overdetermined. One way to design a specific input sequence that is sparse in the Fourier domain is to prescribe a set of interpolation points $\fourierPoint{k_i}$ for $i=1,\ldots,r$ and define
\begin{equation}
	\label{eq:sparseInputSignal}
	\inpVar_j \vcentcolon= \frac{1}{\numTimeSteps}\sum_{i=1}^r \fourierPoint{k_i}^j.
\end{equation}
Then, the \gls{FFT} implies
\begin{displaymath}
	\ZinpVar_{k;\numTimeSteps} = \sum_{j=0}^{\numTimeSteps-1} \inpVar_j\fourierPoint{k}^{-j} = \frac{1}{\numTimeSteps}\sum_{i=1}^r \sum_{j=0}^{\numTimeSteps-1} \exp\left(\frac{2\pi\imath}{\numTimeSteps}(k_i - k)j\right) = \frac{1}{\numTimeSteps} \sum_{i=1}^r \numTimeSteps\delta_{k_i,k},
\end{displaymath}
i.e., only the Fourier coefficients corresponding to the $k_i$ are nonzero. Note that in this case we do not need to compute the \gls{FFT} of $\inpVar_j$ and $\FourierMatrix$ is a generalized Vandermonde matrix.

\section{Implementation Details}
\label{sec:implementationDetailslsETFE}

The results of the previous section are formulated in a discrete-time setting. The following observation allows us to transfer the results of \Cref{sec:lsTFE} to continuous-time systems. Consider the system
\begin{equation*}
	\label{eq:standardLTIsystem}
	\begin{aligned}
		\dot{\state}(t) &= \Aa\state(t) + \B\inpVar(t),\\
		\outVar(t) &= \C\state(t),\\
		\state(0) &= \state_0
	\end{aligned}
\end{equation*}
and the control function
\begin{equation}
	\label{eq:continuousSparseFourierInput}
	\inpVar(t) = \frac{\timeStep}{\tf}\sum_{i=1}^r \exp\left(2\pi\imath k_i \frac{t}{\tf}\right).
\end{equation}
Evaluating $u$ at the time grid $t_j = j\timeStep$ with $j\in\{0,1,\ldots,\numTimeSteps\}$ reveals
\begin{displaymath}
	\inpVar(t_j) = \frac{\timeStep}{N\timeStep}\sum_{i=1}^r \exp\left(2\pi\imath k_i \frac{j\timeStep}{\numTimeSteps \timeStep}\right) = \frac{1}{\numTimeSteps}\sum_{i=1}^r \exp\left(\frac{2\pi \imath}{\numTimeSteps}k_ij\right) = \inpVar_j,
\end{displaymath}
i.e., the input signal in \eqref{eq:continuousSparseFourierInput} can be understood as a continuous representation of the discrete input signal in \eqref{eq:sparseInputSignal}. For $t>0$ we have
\begin{align*}
	\outVar(t) &= \C\int_0^t \exp(\Aa(t-s))\B\inpVar(s)\mathrm{d}s\\
	&= \frac{\C}{\numTimeSteps}\sum_{i=1}^r\left(\frac{2\pi \imath}{\tf}k_i I_{\stateDim} - \Aa\right)^{-1}\exp(\Aa t) \left(\exp\left(\left(\frac{2\pi\imath}{\tf}k_i I_{\stateDim} - \Aa\right) t\right) - I_{\stateDim}\right)\B.
\end{align*}
If we assume that $\Aa$ is asymptotically stable, then for sufficiently large $t$, we have $\exp(\Aa t) \approx 0$ and hence
\begin{displaymath}
	\outVar(t) \approx \frac{1}{\numTimeSteps}\sum_{i=1}^r \transfer\left(\frac{2\pi\imath}{\tf}k_i\right)\exp\left(2\pi\imath k_i \frac{t}{\tf}\right).
\end{displaymath}
A comparison with \eqref{eq:timeFrequencyRelation} suggests that using the the input signal \eqref{eq:continuousSparseFourierInput} in combination with the procedure in \Cref{sec:lsTFE} results in an approximation of the transfer function of the continuous-time system at the frequency $\tfrac{2\pi\imath}{\tf}k_i$. As a consequence, we can describe frequency bounds $\fmin,\fmax>0$ and choose $\widetilde{r}$ interpolation points $\widetilde{\pole}_i$ in the interval $[\imath\fmin,\imath\fmax]$. For a given final time $\tf = \numTimeSteps\timeStep$ and given $i\in\{1,\ldots,\widetilde{r}\}$, we can thus compute the number $k_i\in\mathbb{N}$ that minimizes
\begin{equation}
	\label{eq:transformationContDiscFreq}
	\left|\frac{2\pi\imath }{\tf}k_i-\widetilde{\pole}_i\right| = \min_{k\in\mathbb{N}} \left|\frac{2\pi \imath}{\tf} k-\widetilde{\pole}_i\right|.
\end{equation}
The transfer function is thus estimated at the frequencies
\begin{equation}
	\label{eq:actualFrequencies}
	\pole_i \vcentcolon= \frac{2\pi\imath }{\tf}k_i\qquad\text{for $i\in\{1,\ldots,\widetilde{r}\}$}.
\end{equation}
Note that for some choices of $\widetilde{\pole}_i$, we may have $\pole_i = \pole_j$ for $i\neq j$. Thus, we remove redundant frequencies to obtain $r$ unique frequencies.
These frequencies are related to the $\fourierPoint{k_i}$ via
\begin{equation}
	\label{eq:qLambda}
	\fourierPoint{k_i} = \exp(\pole_i\timeStep).
\end{equation}
We summarize the previous discussion and the results of \Cref{sec:lsTFE} in \Cref{alg:LSTFE}.

\begin{algorithm}[ht]
	\caption{Least-Squares Transfer Function Estimate}
	\label{alg:LSTFE}
	\begin{algorithmic}[1]
		\Statex \textbf{Input:} $\jmin$ and desired interpolation frequencies $\widetilde{\pole}_i$ ($i=1,\ldots,\widetilde{r}$)
		\Statex \textbf{Output:} actual frequencies $\pole_i$ ($i=1,\ldots,r$) together with estimates of the transfer function at these frequencies
		\Statex
		\State Solve the minimization problem \eqref{eq:transformationContDiscFreq} for $i=1,\ldots,\widetilde{r}$
		\State Remove redundant frequencies to obtain unique frequencies $\pole_i$ according to \eqref{eq:actualFrequencies} and corresponding points $\fourierPoint{k_i}$, cf.\ \eqref{eq:qLambda}, for $i=1,\ldots,r$
		\State Construct the input signal $\inpVar_j$ according to \eqref{eq:sparseInputSignal} and obtain measurements $\outVar_j$
		\State Compute the Fourier coefficients of $\inpVar_j$ and assemble the matrix $\FourierMatrix$
		\State Solve the regularized minimization problem \eqref{eq:regularizedLeastSquaresETFE}
	\end{algorithmic}
\end{algorithm}

In our examples, we use a logarithmic sampling of the frequency interval $[\imath\fmin,\imath\fmax]$ and pick $\jmin$ such that \SI{75}{\percent} of the time series is used for the least-squares problem \eqref{eq:leastSquaresETFE}. For details about the choice of $\jmin$ we refer to \cite{PehGW17}.

\section{Structured Interpolation Framework}
\label{sec:structuredRealization}

The term structured \gls{LTI} operator in \Cref{problem:generalSetting} may in general have wide-ranging meanings. Here, according to the examples given in \Cref{tab:StructureExamples}, we consider a structure which is induced by a linearly independent function family $\{\hfunc_1,\, \hfunc_2,\, \ldots,\, \hfunc_\numFunctions\}$ and takes the form
\begin{equation}
	\label{eq:FOMstructured}
	\transfer(\frequency) = \C\left(\sum_{k=1}^\numFunctions \hfunc_k(\frequency) \A_k\right)^{-1}\B,
\end{equation}
where $\C\in\mat{\mathbb{R}}{1}{\stateDim}$, $A_k\in\mat{\mathbb{R}}{\stateDim}{\stateDim}$ for $k=1,\ldots,\numFunctions$, and $B\in\mat{\mathbb{R}}{\stateDim}{1}$. We assume in all that follows that the functions $\hfunc_k\colon\mathbb{C}\to\mathbb{C}$ are meromorphic and by standard abuse of notation, we use $\transfer(\frequency)$ to denote either the system itself or the transfer function of the system evaluated at the point $\frequency\in\mathbb{C}$. 

Following \cite{SchUBG18}, we can enforce $\numFunctions\numData$ interpolation conditions for the structure \eqref{eq:FOMstructured} and hence, we assume that our estimate of the transfer function provides $\numFunctions\numData$ distinct points in the complex plane. More precisely, we assume to have the \emph{interpolation data}
\begin{equation}
	\label{eq:interpolationData}
	\{\left(\pole_i,\interpData_i \vcentcolon=\transfer(\pole_i)\right)\in\mathbb{C}^2 \mid i=1,\ldots,\numFunctions\numData\}
\end{equation}
available.

\begin{problem}[Structured realization problem]
	\label{problem:StructuredRealization}
	Given the data in \eqref{eq:interpolationData} and a system structure associated with the linearly independent function family $\{\hfunc_1,\ldots,\hfunc_\numFunctions\}$, find matrices $\Ared_k\in\mat{\mathbb{C}}{\stateDim}{\stateDim}$, $k=1,\ldots,\numFunctions$, $\Bred\in\mat{\mathbb{C}}{\stateDim}{1}$, and $\Cred\in\mat{\mathbb{C}}{1}{\stateDim}$, such that the transfer function
	\begin{equation*}
		\label{eq:structuredTransferFunction}
		\transferRed(\frequency) = \Cred\left(\sum_{k=1}^{\numFunctions} \hfunc_k(\frequency)\Ared_k\right)^{-1}\Bred
	\end{equation*}
	satisfies the interpolation conditions
	\begin{equation}
		\label{eq:interpolationCondition}
		\transferRed(\pole_i) = \interpData_i \qquad\text{for } i=1,\ldots,\numFunctions\numData.
	\end{equation}
\end{problem}

In general we cannot expect that the realization in \Cref{problem:StructuredRealization} is real, i.e., that all the matrices $\Ared_k$, $\Bred$, and $\Cred$ are real matrices. If we however add the complex conjugate points $(\overline{\pole}_i,\overline{\interpData}_i)$ to the interpolation data \eqref{eq:interpolationData}, then \Cref{problem:StructuredRealization} can be solved with real matrices \cite{SchUBG18}. Since the \gls{ETFE} framework provides only points on the imaginary axis, we simply add the complex conjugate data and assume $\numData$ to be an even number and
\begin{equation}
	\label{eq:closedConjugation}
	(\overline{\pole_{2i-1}},\overline{\interpData_{2i-1}}) = (\pole_{2i},\interpData_{2i})\qquad \text{for } i=1,\ldots,\frac{\numFunctions\numData}{2}.
\end{equation}

\Cref{problem:StructuredRealization} was solved in a more general setting in \cite{SchUBG18} and to make the presentation self-contained, we recall the important results tailored to our specific setting. Let $Q_{\LeftData} \vcentcolon= \ceil*{\tfrac{\numFunctions}{2}}$, $Q_{\RightData} = \floor*{\tfrac{\numFunctions}{2}}$ such that $Q_{\LeftData} + Q_{\RightData} = \numFunctions$. Moreover, we rename the interpolation data as
\begin{align*}
	\leftPoint_{j;i} &\vcentcolon= \pole_{2(j-1)\numData+i}, & \leftData_{j;i} &\vcentcolon= \interpData_{2(j-1)\numData+i}, & \text{for } j&=1,\ldots,Q_{\LeftData},\ i=1,\ldots,\numData,\\
	\rightPoint_{j;i} &\vcentcolon= \pole_{(2j-1)\numData+i}, & \rightData_{j;i} &\vcentcolon= \interpData_{(2j-1)\numData+i}, & \text{for } j&=1,\ldots,Q_{\RightData},\ i=1,\ldots,\numData,
\end{align*}
and define the matrix
\begin{equation*}
	T = \mathrm{blkdiag}\left(\frac{1}{\sqrt{2}}\begin{bmatrix}
1 & -\imath\\ 1 & \imath
\end{bmatrix},\,\ldots,\,\frac{1}{\sqrt{2}}\begin{bmatrix}
1 & -\imath\\ 1 & \imath
\end{bmatrix}\right)\in\mat{\mathbb{C}}{\numData}{\numData}.
\end{equation*}

\begin{theorem}
	\label{thm:additionalInterpolation}
	Assume that the functions $\hfunc_k$ satisfy the Haar condition \cite{Che82}. the interpolation data \eqref{eq:interpolationData} satisfies \eqref{eq:closedConjugation}, and $\interpData_i \neq 0$ for all $i=1,\ldots,\numFunctions\numData$. Define the matrices $\A_k\in\mat{\mathbb{C}}{\numData}{\numData}$ entry-wise via the linear systems
	{\small
	\begin{equation}
	\label{eq:HaarSystem}
	\begin{bmatrix}
		\leftData_{1;i} & & & & &\\
		& \ddots &  & & &\\
		& & \leftData_{Q_\LeftData;i} & & &\\
		& & & \rightData_{1;j} & & \\
		& & & & \ddots & \\
		& & & & & \rightData_{Q_\RightData;j}
	\end{bmatrix}\begin{bmatrix}
		\hfunc_1(\leftPoint_{1;i}) & \ldots & \hfunc_\numFunctions(\leftPoint_{1;i})\\
		\vdots & & \vdots\\
		\hfunc_1(\leftPoint_{Q_{\LeftData};i}) & \ldots & \hfunc_\numFunctions(\leftPoint_{Q_{\LeftData};i}) \\
		\hfunc_1(\rightPoint_{1;j}) & \ldots & \hfunc_\numFunctions(\rightPoint_{1;j})\\
		\vdots & & \vdots\\
		\hfunc_1(\rightPoint_{Q_{\RightData};j}) & \ldots & \hfunc_\numFunctions(\rightPoint_{Q_{\RightData};j})
	\end{bmatrix}\begin{bmatrix}
		[\A_1]_{i,j}\\
		[\A_2]_{i,j}\\
		\vdots\\
		[\A_\numFunctions]_{i,j}
	\end{bmatrix} = \begin{bmatrix}
		1\\1\\\vdots\\1
	\end{bmatrix}.
	\end{equation}}%
	Then the matrices $\Ared_k \vcentcolon= T^*\A_kT$ for $k=1,\ldots,\numFunctions$, $\Bred \vcentcolon= T^* \begin{bmatrix}
		1 & \ldots & 1
	\end{bmatrix}^T$, and $\Cred \vcentcolon= \Bred^T$ are real and the structured realization satisfies the interpolation conditions, i.e.,
	\begin{displaymath}
		\transferRed(\pole_i) \vcentcolon= \Cred\left(\sum_{k=1}^\numFunctions \hfunc_k(\pole_i)\Ared\right)^{-1}\Bred = \interpData_i\qquad \text{for}\ i=1,\ldots,\numFunctions\numData,
	\end{displaymath}
	provided that $\sum_{k=1}^\numFunctions \hfunc_k(\pole_i)\Ared$ is nonsingular for $i=1,\ldots,\numFunctions\numData$.
\end{theorem}

\begin{proof}
	The Haar condition together with $\interpData_i\neq 0$ for $i=1,\ldots,\numFunctions\numData$ ensures that the linear system \eqref{eq:HaarSystem} has a unique solution. The result follows from \cite[Theorem~3.12 and Lemma~3.15]{SchUBG18}.
\end{proof}

An important aspect in \Cref{thm:additionalInterpolation} is that $\Kred(s) \vcentcolon= \sum_{k=1}^\numFunctions \hfunc_k(s)\Ared$ is (numerically) nonsingular at all interpolation points $\pole_i$, which may not be true if we add more and more data. If the condition
\begin{equation}
	\label{eq:redundantAss}
	\mathrm{rank} \left(\sum_{k=1}^\numFunctions \hfunc_k(\pole_i)\Ared_k\right) = \mathrm{rank}\left(\begin{bmatrix}
		\Ared_1 & \cdots & \Ared_\numFunctions
	\end{bmatrix}\right) = \mathrm{rank}\left(\begin{bmatrix}
		\Ared_1 \\ \vdots \\ \Ared_\numFunctions
	\end{bmatrix}\right) =\vcentcolon r.
\end{equation}
is satisfies for all driving frequencies $\pole_i$, then the redundancy in the data can be removed as follows.

\begin{theorem}
	\label{thm:truncationAdditionalData}
	Let the realization $\Hred(s) = \Cred(\sum_{k=1}^\numFunctions \hfunc_k(s)\Ared_k)^{-1}\Bred$ be constructed as in \Cref{thm:additionalInterpolation} and assume that the matrices satisfy the rank condition. Pick any $i\in\{1,\ldots,\numFunctions\numData\}$ and let
	\begin{equation}
		\label{eq:truncationSVD}
		\Kred(\pole_i) = \sum_{k=1}^\numFunctions \hfunc_k(\pole_i)\Ared_k = \begin{bmatrix} W_1 & W_2\end{bmatrix}\begin{bmatrix}\Sigma & 0\\
		0 & 0
		\end{bmatrix}
		\begin{bmatrix}
			V_1^*\\
			V_2^*
		\end{bmatrix} 
	\end{equation}
	with $V_1,W_1\in\mat{\mathbb{C}}{\numData}{r}$, $V_2,W_2\in\mat{\mathbb{C}}{\numData}{(\numData-r)}$, and nonsingular matrix $\Sigma\in\mathbb{R}^{r\times r}$ denote the \gls{SVD} of $\Kred(\pole_i)$. For $k=1,\ldots,\numFunctions$ define
	\begin{displaymath}
		\Ared_{k;r} \vcentcolon= W_1^*\Ared_k V_1,\qquad \Bred_r \vcentcolon= W_1^*\Bred, \qquad\text{and}\qquad \Cred_r \vcentcolon = \Cred V_1.
	\end{displaymath}	
	Then the realization $\Hred_r(s) = \Cred_r(\sum_{k=1}^\numFunctions \hfunc_k\left(s\right)\Ared_{k;r})^{-1}\Bred_r$ satisfies the interpolation conditions \eqref{eq:interpolationCondition}.
\end{theorem}

\begin{proof}
	The rank assumptions \eqref{eq:redundantAss} and
	\begin{displaymath}
		\ker\left(\begin{bmatrix}
			\Ared_1^* & \cdots & \Ared_K^*
		\end{bmatrix}^*\right) \subseteq \ker(\sum_{k=1}^K\hfunc_k(s)\Ared_k)
	\end{displaymath}
	imply $\Ared_kV_2 = 0$ for all $k=1,\ldots,\numFunctions$. By the same reasoning we obtain $W_2^*\Ared_k = 0$ for all $k=1,\ldots,\numFunctions$. Then, the result follows using \cite[Theorem~3.19]{SchUBG18}.
\end{proof}

Summarizing \Cref{thm:additionalInterpolation} and \Cref{thm:truncationAdditionalData}, we can construct a structured realization, i.e., solve \Cref{problem:StructuredRealization}, via \Cref{alg:structuredRealization}.

\begin{algorithm}[ht]
	\caption{Structured Realization}
	\label{alg:structuredRealization}
	\begin{algorithmic}[1]
		\Statex \textbf{Input:} Interpolation data \eqref{eq:interpolationData}, function family $\lbrace\hfunc_1,\ldots,\hfunc_\numFunctions\rbrace$ with $\numFunctions\in\mathbb{N}$.
		\Statex \textbf{Output:} Matrices $\Ared_1,\ldots,\Ared_\numFunctions$, $\Bred$, and $\Cred$ such that $\Hred(s) = \Cred(\sum_{k=1}^\numFunctions h(s)\Ared_k)^{-1}\Bred$ interpolates the data
		\Statex
		\State Add the complex conjugate data as in \eqref{eq:closedConjugation}.
		\State Construct the realization $\Hred(s) = \Cred(\sum_{k=1}^\numFunctions h(s)\Ared_k)^{-1}\Bred$ as in \Cref{thm:additionalInterpolation} by solving the linear systems \eqref{eq:HaarSystem}.
		\State Pick any $i\in\{1,\ldots,\numFunctions\numData\}$ and compute $V_1,W_1$ via the singular value decomposition as in \eqref{eq:truncationSVD}.
		\State Set $\Ared_k \vcentcolon= W_1^*\Ared_kV_1$, $\Bred \vcentcolon= W_1^*\Bred$, and $\Cred \vcentcolon= \Cred V_1$
	\end{algorithmic}
\end{algorithm}

\section{Estimation of Parameters}
\label{sec:parameterEstimation}

Comparing \Cref{problem:StructuredRealization} with the structured realizations in \Cref{tab:StructureExamples}, we observe that the coefficient functions $\hfunc_k$ may depend on possibly unknown parameters like the time delay~$\delay$, which also need to be identified. Let us thus consider a linearly independent function family
\begin{displaymath}
	\hfunc_k:\mathbb{C}\times\paramSet \to\mathbb{C}\qquad\text{for}\ k=1,\ldots,\numFunctions
\end{displaymath}
with a compact parameter set $\paramSet\subseteq\mathbb{R}^{\paramSetDim}$. Observe that for any fixed $\param\in\paramSet$ we can use \Cref{alg:structuredRealization} to obtain a realization 
\begin{equation}
	\label{eq:pROM}
	\Hred(s,\param) = \Cred(\param)\left(\sum_{k=1}^{\numFunctions} \hfunc_k(s,\param)\Ared_k(\param)\right)^{-1}\Bred(\param)
\end{equation}
which interpolates the data for this specific parameter. If further data
\begin{equation}
	\label{eq:testData}
	\{\left(\zeta_j,\psi_j\vcentcolon= \transfer(\zeta_j)\right)\in\mathbb{C}^2\mid j=1,\ldots,\numTestData\}
\end{equation}
are available, in the following referred to as \emph{test data}, we can compute the least-squares mismatch
\begin{equation}
	\label{eq:lsErrorFunctional}
	\lsError:\paramSet\to\mathbb{R},\qquad \param \mapsto \sum_{j=1}^\numTestData\left\|\psi_j - \Hred(\zeta_j,\param)\right\|^2
\end{equation}
between evaluations of the transfer function \eqref{eq:pROM} and this data. A simple strategy, as for instance proposed in \cite{SchU16}, is to minimize \eqref{eq:lsErrorFunctional} over the parameter set $\paramSet$. Note that if an optimal parameter $\param^\star\in\paramSet$ is determined, one can add the test data \eqref{eq:testData} to the interpolation data \eqref{eq:interpolationData} and compute a realization that interpolates also the test data via \Cref{alg:structuredRealization}.

\section{A Case Study with a Delay Example}
\label{sec:caseStudy}

To illustrate the framework presented in this paper, we consider the delay example from \cite{BeaG09}. More precisely, we consider the \gls{DDAE}
\begin{align*}
	\E\dot{\state}(t) &= \Aa\state(t) + \Ab\state(t-\delay) + \B\inpVar(t), & \text{for $t>0$},\\
	\outVar(t) &= \C\state(t), & \text{for $t>0$},\\
	\state(t) &= 0, & \text{for $t\in[-\delay,0]$},
\end{align*}
with $\dimFOM\times\dimFOM$ matrices
\begin{align*}
	\E &\vcentcolon= \nu I_{\dimFOM} + T, & \Aa &\vcentcolon= \frac{1}{\delay}\left(\frac{1}{\zeta}+1\right)(T-\nu I_{\dimFOM}), & \Ab &\vcentcolon= \frac{1}{\delay}\left(\frac{1}{\zeta}-1\right)(T-\nu I_{\dimFOM}),
\end{align*}	
where $T$ is an $\dimFOM\times\dimFOM$ matrix with ones on the sub- and superdiagonal, at the $(1,1)$, and at the $(\dimFOM,\dimFOM)$ position and zeros everywhere else. We choose $\dimFOM = 12$, $\delay=1$, $\zeta=0.01$, and $\nu = 5$. The input matrix $B\in\mathbb{R}^{\dimFOM}$ has ones in the first two components and zeros everywhere else, and we choose $C = B^T$. 

We simulate the model twice to obtain estimates of the transfer function: once for setting up the initial model (i.e., for collecting the interpolation data) and once for obtaining additional test data, such that we can estimate the delay via minimizing the least-squares mismatch in \eqref{eq:lsErrorFunctional}. The simulation parameters are listed in \Cref{tab:lsETFEsimulationParameters}. Note that we use higher frequencies to construct the input function for the test data than for the input function for the interpolation data. These higher frequencies enforce a smaller time step $\timeStep$. In order to have a similar computational cost for both input functions, we therefore adapted the final time $\tf$. Consequently, \Cref{thm:convergenceETFE} suggests that we can expect a better accuracy for the transfer function estimates obtained from the simulation used for the interpolation data.

The resulting transfer function estimates are compared to the true values in \Cref{tab:lsETFEinterpolationData,tab:lsETFEtestData} and visualized in \Cref{fig:lsETFE}. Note that in this section, all numerical values are rounded to two decimal places.
\begin{table}
	\centering
	\caption{Simulation parameters to obtain the transfer function estimates via lsTFE}
	\label{tab:lsETFEsimulationParameters}
	\footnotesize
	\begin{tabular}{llll}
		\toprule
		\textbf{description} & \textbf{variable} & \textbf{interpolation data} & \textbf{test data}\\\midrule
		final time & $\tf$ & 10000 & 40\\
		time step & $\timeStep$ & \num{5e-3} & \num{1e-5}\\
		frequency sampling interval & $[\fmin,\fmax]$ & $[$\num{1e-4}$,$\num{1e0}$]$ & $[$\num{1e0.3},\num{1e1}$]$\\
		requested number of frequency estimates & $\widetilde{r}$ & 10 & 6\\
		actual number of frequency estimates & $r$ & 8 & 6\\\bottomrule
	\end{tabular}
\end{table}
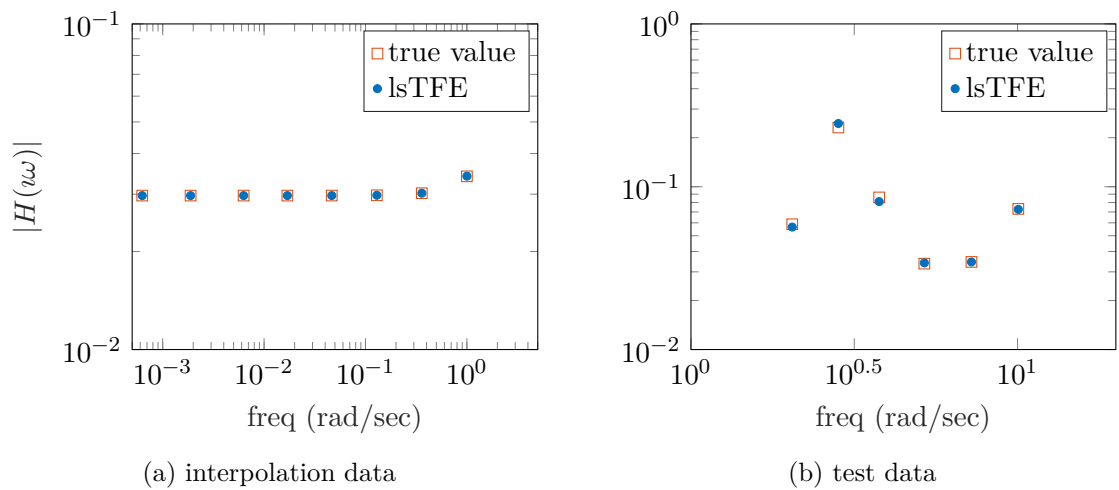
\begin{figure}
	\centering
	\begin{subfigure}[c]{0.48\textwidth}
		\centering
%
\definecolor{mycolor1}{rgb}{0.00000,0.44700,0.74100}%
\definecolor{mycolor2}{rgb}{0.85000,0.32500,0.09800}%
\begin{tikzpicture}

\begin{axis}[%
width=2.1in,
height=1.7in,
at={(0.758in,0.481in)},
scale only axis,
xmode=log,
xmin=0.0005,
xmax=5,
xminorticks=true,
xlabel style={font=\color{white!15!black}},
xlabel={freq (rad/sec)},
ymode=log,
ymode=log,
ymin=0.01,
ymax=0.1,
ylabel style={font=\color{white!15!black}},
ylabel={$|\transfer(\imath\omega)|$},
axis background/.style={fill=white},
legend style={legend cell align=left, align=left, draw=white!15!black}
]

\addplot [only marks,color=mycolor2, draw=none, mark=square, mark options={solid, mycolor2}]
  table[row sep=crcr]{%
0.000628318530717959	0.0297032713067957\\
0.00188495559215388	0.0297032835014317\\
0.00628318530717958	0.0297034222160043\\
0.0169646003293849	0.029704381052865\\
0.0464955712731289	0.0297116189636064\\
0.129433617327899	0.0297680738695372\\
0.359398199570672	0.030209084896221\\
1.00028310090299	0.0340351879760734\\
};
\addlegendentry{true value}

\addplot [only marks,color=mycolor1, draw=none, mark=*, mark options={solid, mycolor1},mark size=1.5pt]
  table[row sep=crcr]{%
0.000628318530717959	0.0297032713214858\\
0.00188495559215388	0.0297032836336428\\
0.00628318530717958	0.0297034236850327\\
0.0169646003293849	0.0297043917626637\\
0.0464955712731289	0.0297116994450372\\
0.129433617327899	0.0297686995569023\\
0.359398199570672	0.0302140302616045\\
1.00028310090299	0.0340821078153398\\
};
\addlegendentry{lsTFE}

\end{axis}
\end{tikzpicture}%
		\subcaption{interpolation data}
		\label{fig:lsETFE-interpolationData}
	\end{subfigure}\hfill
	\begin{subfigure}[c]{0.48\textwidth}
		\centering
%
\definecolor{mycolor1}{rgb}{0.00000,0.44700,0.74100}%
\definecolor{mycolor2}{rgb}{0.85000,0.32500,0.09800}%
\begin{tikzpicture}

\begin{axis}[%
width=2.2in,
height=1.7in,
at={(0.758in,0.481in)},
scale only axis,
xmode=log,
xmin=1,
xmax=20,
xminorticks=true,
xlabel style={font=\color{white!15!black}},
xlabel={freq (rad/sec)},
ymode=log,
ymin=0.01,
ymax=1,
yminorticks=true,
axis background/.style={fill=white},
legend style={legend cell align=left, align=left, draw=white!15!black}
]

\addplot [only marks,color=mycolor2, draw=none, mark=square, mark options={solid, mycolor2}]
  table[row sep=crcr]{%
2.04203522483337	0.0588110012338469\\
2.82743338823081	0.231128623467937\\
3.76991118430775	0.0859603000185347\\
5.18362787842316	0.033704823124924\\
7.22566310325653	0.0345183235467544\\
10.0530964914873	0.0730740859910882\\
};
\addlegendentry{true value}

\addplot [only marks,color=mycolor1, draw=none, mark=*, mark options={solid, mycolor1},mark size=1.5pt]
  table[row sep=crcr]{%
2.04203522483337	0.0565237336951213\\
2.82743338823081	0.244544453960406\\
3.76991118430775	0.08104315062105\\
5.18362787842316	0.0340606789161688\\
7.22566310325653	0.0345068320897127\\
10.0530964914873	0.0725969453104999\\
};
\addlegendentry{lsTFE}

\end{axis}
\end{tikzpicture}%
		\subcaption{test data}
		\label{fig:lsETFE-testData}
	\end{subfigure}
	\caption{Estimation of the transfer function via lsTFE. The estimates are plotted with blue dots and the true values of the transfer function with red squares.}
	\label{fig:lsETFE}
\end{figure}
\begin{table}
	\centering
	\caption{Interpolation data: estimates of the transfer function via \gls{lsETFE}}
	\label{tab:lsETFEinterpolationData}
	\footnotesize
	\begin{tabular}{llll}
		\toprule
		\textbf{frequency} $\omega$ & \textbf{true value} $\transfer(\imath\omega)$ & \textbf{\gls{lsETFE} estimate} $\transferEstimate_{k_1;\numTimeSteps}$ & \textbf{error}\\\midrule
		\num{6.28e-04} & \num{2.97e-02} + $\imath$\num{9.05e-06} & \num{2.97e-02} + $\imath$\num{9.00e-06} & \num{4.71e-08}\\
\num{1.88e-03} & \num{2.97e-02} + $\imath$\num{2.71e-05} & \num{2.97e-02} + $\imath$\num{2.70e-05} & \num{1.41e-07}\\
\num{6.28e-03} & \num{2.97e-02} + $\imath$\num{9.05e-05} & \num{2.97e-02} + $\imath$\num{9.00e-05} & \num{4.71e-07}\\
\num{1.70e-02} & \num{2.97e-02} + $\imath$\num{2.44e-04} & \num{2.97e-02} + $\imath$\num{2.43e-04} & \num{1.27e-06}\\
\num{4.65e-02} & \num{2.97e-02} + $\imath$\num{6.70e-04} & \num{2.97e-02} + $\imath$\num{6.66e-04} & \num{3.49e-06}\\
\num{1.29e-01} & \num{2.97e-02} + $\imath$\num{1.87e-03} & \num{2.97e-02} + $\imath$\num{1.86e-03} & \num{9.75e-06}\\
\num{3.59e-01} & \num{2.98e-02} + $\imath$\num{5.24e-03} & \num{2.98e-02} + $\imath$\num{5.21e-03} & \num{2.79e-05}\\
\num{1.00e+00} & \num{3.01e-02} + $\imath$\num{1.58e-02} & \num{3.02e-02} + $\imath$\num{1.58e-02} & \num{9.86e-05}\\\bottomrule
	\end{tabular}
\end{table}
The approximation of the transfer function at the lower frequencies (cf.\ \Cref{fig:lsETFE-interpolationData}) is almost matching the true values. Indeed, the maximum error between the estimates and the true values of the transfer function in the interpolation data set is \num{9.86e-5}. The approximation for the higher frequencies (cf.\ \Cref{fig:lsETFE-testData} and \Cref{tab:lsETFEtestData}) is -- as expected -- significantly worse. However, even in this frequency range, the approximation is reasonable with a maximum error of \num{1.47e-2}.
\begin{table}
	\centering
	\caption{Test data: estimates of the transfer function via \gls{lsETFE}}
	\label{tab:lsETFEtestData}
	\footnotesize
	\begin{tabular}{llll}
		\toprule
		\textbf{frequency} $\omega$ & \textbf{true value} $\transfer(\imath\omega)$ & \textbf{\gls{lsETFE} estimate} $\transferEstimate_{k_1;\numTimeSteps}$ & \textbf{error}\\\midrule
		\num{2.04e+00} & \num{3.26e-02} + $\imath$\num{4.89e-02} & \num{3.11e-02} + $\imath$\num{4.72e-02} & \num{2.32e-03}\\
\num{2.83e+00} & \num{6.16e-02} + $\imath$\num{2.23e-01} & \num{5.93e-02} + $\imath$\num{2.37e-01} & \num{1.47e-02}\\
\num{3.77e+00} & \num{2.60e-02} - $\imath$\num{8.19e-02} & \num{2.06e-02} - $\imath$\num{7.84e-02} & \num{6.48e-03}\\
\num{5.18e+00} & \num{2.79e-02} - $\imath$\num{1.89e-02} & \num{2.89e-02} - $\imath$\num{1.80e-02} & \num{1.35e-03}\\
\num{7.23e+00} & \num{3.19e-02} + $\imath$\num{1.31e-02} & \num{3.23e-02} + $\imath$\num{1.21e-02} & \num{1.07e-03}\\
\num{1.01e+01} & \num{1.88e-02} - $\imath$\num{7.06e-02} & \num{1.84e-02} - $\imath$\num{7.02e-02} & \num{5.80e-04}\\\bottomrule
	\end{tabular}
\end{table}

Before we can apply \Cref{alg:structuredRealization}, we need to specify the structure via defining the function family $\lbrace \hfunc_1,\ldots,\hfunc_\numFunctions\rbrace$. To this end, we first use the actual structure of the original model, i.e., the ${\hfunc_k}$'s are given by 
\begin{displaymath}
	\hfunc_1(s) = s, \qquad \hfunc_2(s) \equiv -1,\qquad \text{and}\qquad \hfunc_3(s) = -\mathrm{e}^{-s}.
\end{displaymath}
Note that the choice for $\hfunc_3$ includes the true value for the delay time $\delay = 1$.
To obtain a real realization, we add the complex conjugate data to the estimated transfer function values given in \Cref{tab:lsETFEinterpolationData} and choose $Q_{\LeftData}=2$ and $Q_{\RightData} = 1$. The transfer function of the obtained realization is depicted as the red dashed line in \Cref{fig:lsETFE-FOMvsROM}. Although we use only approximations of the transfer function, the realization approximates the original model (blue solid line in \Cref{fig:lsETFE-FOMvsROM}) well, even for frequencies larger than the frequencies used to construct the realization.

\begin{figure}
	\centering
	\input{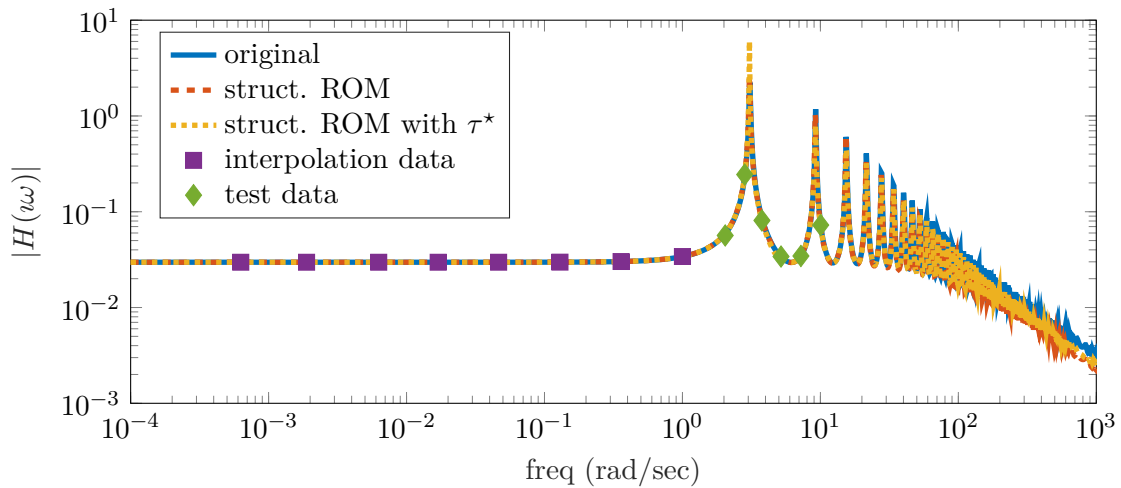}
	\caption{The transfer function of the true model (solid blue line), the realization (dashed red line), and the realization with estimated parameter $\delay^\star$ (dotted yellow line) obtained from the estimated interpolation data (orange squares) and test data (green diamonds).}
	\label{fig:lsETFE-FOMvsROM}
\end{figure}

In a real application, we usually cannot compare the transfer function of the constructed realization with the true transfer function, since we do not know the true model and hence also do not know the transfer function. Instead, it is more reasonable to compare the realization with the true model via simulations in the time domain (see also \Cref{problem:generalSetting}). As validation input functions we use
\begin{align*}
	\inpVar_1(t) &= \sin(t), & \inpVar_2(t) &= 2\left(t-\tfrac{1}{2}\floor{2t+\tfrac{1}{2}}\right)\cdot(-1)^{\floor{2t+\frac{1}{2}}} + 1, & \inpVar_3(t) &= t\exp(-t^2). 
\end{align*}
The results are presented in \Cref{fig:timeSeries-ROMvsFOMsimulation} and \Cref{tab:errorMeasurementsForValidationInputs}.
The relative errors given in \Cref{tab:errorMeasurementsForValidationInputs} indicate slight differences between the accuracies obtained for the three different input signals. 
Nevertheless, the output trajectories of the realization agree very well with the ones of the original model for all three inputs, as illustrated in \Cref{fig:timeSeries-ROMvsFOMsimulation}.

\begin{table}
	\centering
	\caption{Error measurements for the validation inputs}
	\label{tab:errorMeasurementsForValidationInputs}
	\begin{tabular}{llll}
		\toprule
		\textbf{input signal} & $\LtwoNorm{\inpVar}$ & $\frac{\LinftyNorm{\outVar-\outVarRed}}{\LtwoNorm{\inpVar}}$ & $\frac{\LtwoNorm{\outVar-\outVarRed}}{\LtwoNorm{\inpVar}}$\\
		\midrule
		$\inpVar_1$ & \num{2.18e+00} & \num{6.96e-04} & \num{1.26e-03}\\
		$\inpVar_2$ & \num{3.29e+00} & \num{3.73e-03} & \num{3.51e-03}\\
		$\inpVar_3$ & \num{3.96e-01} & \num{7.69e-03} & \num{1.01e-02}\\\bottomrule
	\end{tabular}
\end{table}

\begin{figure}[t]
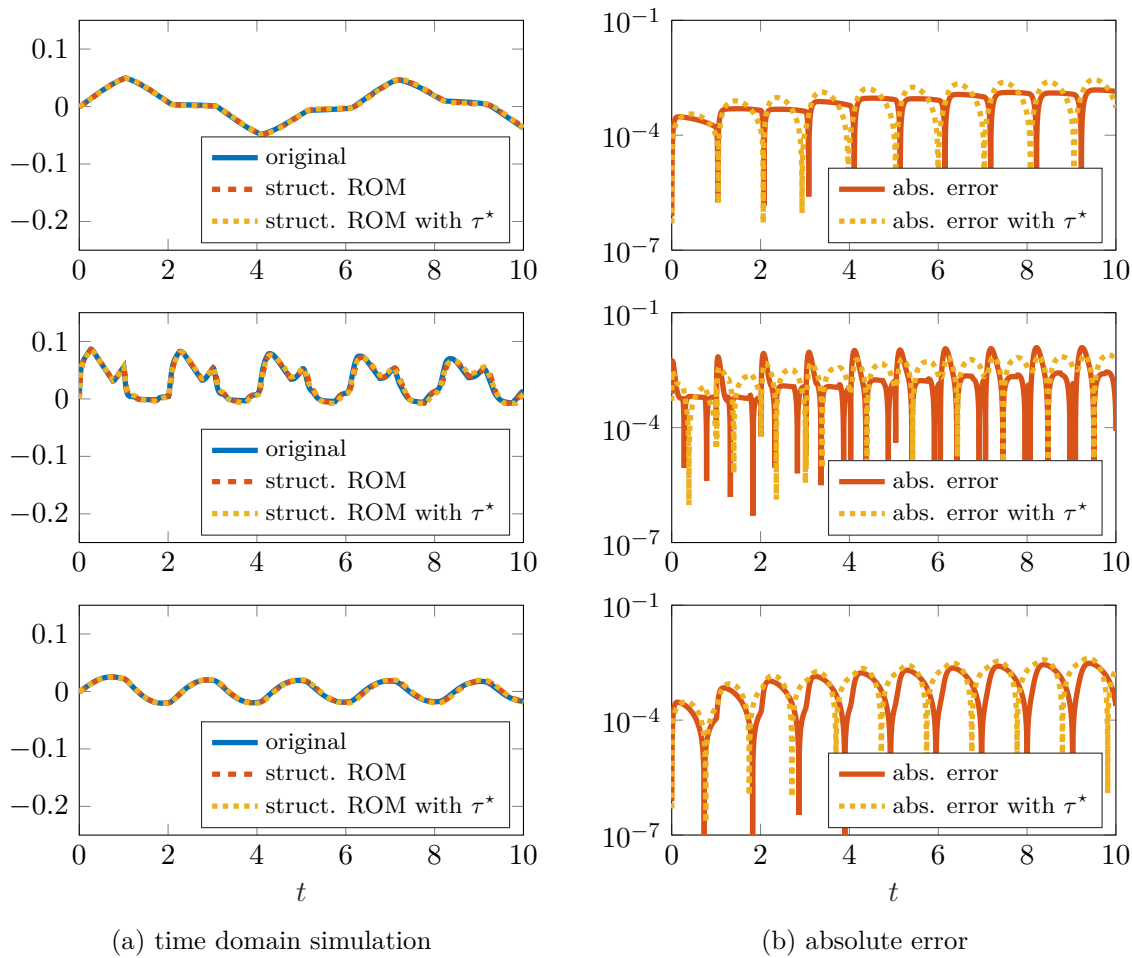

	\begin{subfigure}[b]{.48\linewidth}
		\centering
%
\definecolor{mycolor1}{rgb}{0.00000,0.44700,0.74100}%
\definecolor{mycolor2}{rgb}{0.85000,0.32500,0.09800}%
\definecolor{mycolor3}{rgb}{0.92900,0.69400,0.12500}%
\begin{tikzpicture}

\begin{axis}[%
width=2.3in,
height=1.2in,
at={(0.758in,2.554in)},
scale only axis,
xmin=0,
xmax=10,
ymin=-0.25,
ymax=0.15,
axis background/.style={fill=white},
legend style={at={(0.97,0.03)},font=\footnotesize,anchor=south east,legend cell align=left,align=left,draw=white!15!black}
]

\addplot [color=mycolor1,line width=2pt]
  table[row sep=crcr]{%
0	0\\
0.0289999999999999	0.000821771222900836\\
0.106	0.00504643116301295\\
0.489000000000001	0.0265669916109808\\
0.712999999999999	0.0375566067741193\\
0.907	0.0455755149045416\\
1.032	0.0494693981272167\\
1.08	0.0489031010687686\\
1.186	0.0457446886690036\\
1.371	0.038739587237707\\
1.578	0.029344050936265\\
1.824	0.0165855897950546\\
2.071	0.00389805799074239\\
2.142	0.00327193378472757\\
2.761	0.00225984384759848\\
3.073	0.000774481536994998\\
3.133	-0.0017188799184904\\
3.277	-0.00981357340521782\\
3.593	-0.0271850358029475\\
3.815	-0.0378737329899721\\
4.009	-0.0457081957429892\\
4.087	-0.0477863883156218\\
4.145	-0.0476194117962745\\
4.235	-0.0456036609749049\\
4.404	-0.0401515758415343\\
4.605	-0.0321543696239441\\
4.836	-0.0213868551413796\\
5.136	-0.00727537113225019\\
5.215	-0.00639591848499066\\
5.537	-0.00561865092745428\\
6.066	-0.00331587090956909\\
6.142	-0.00185602641182037\\
6.214	0.00120524493035745\\
6.369	0.0098032247166735\\
6.688	0.0271523337354331\\
6.91	0.0377078735608354\\
7.089	0.044771797677754\\
7.161	0.0462539902827483\\
7.23	0.0460313249168802\\
7.333	0.0439700690521665\\
7.507	0.0388498959229384\\
7.711	0.031327265989292\\
7.944	0.0211662829145389\\
8.184	0.0107651835088589\\
8.264	0.00951140142613838\\
8.433	0.00887788992945815\\
8.877	0.00695817070436178\\
9.161	0.00448440315317988\\
9.237	0.00216200621383322\\
9.332	-0.00244210726598659\\
9.95	-0.0349526441218089\\
10	-0.0371685458193518\\
};
\addlegendentry{\original}

\addplot [color=mycolor2,dashed,line width=2pt]
  table[row sep=crcr]{%
0	0\\
0.032	0.000818990908353356\\
0.103	0.00458056945656438\\
0.534000000000001	0.0286411673392344\\
0.750999999999999	0.0390297201903209\\
0.942	0.0466718225982348\\
1.042	0.0495111957175567\\
1.094	0.0489055753468151\\
1.196	0.0458777977802676\\
1.379	0.0388941635751259\\
1.586	0.0294438373865074\\
2.082	0.00367509689054657\\
2.151	0.00286578608057297\\
3.066	0.000795687632933806\\
3.124	-0.000908695992718478\\
3.208	-0.00514407936858596\\
3.707	-0.0321769497929534\\
3.914	-0.0414339556250631\\
4.078	-0.0473524381088453\\
4.137	-0.0480037021658699\\
4.206	-0.047110008487822\\
4.323	-0.0438378098788892\\
4.51	-0.0370421161475747\\
4.723	-0.027750788862944\\
5.142	-0.00718533396910992\\
5.213	-0.00581487002512837\\
5.342	-0.00519939707355199\\
5.993	-0.00294107693549073\\
6.156	-0.00140546214703008\\
6.227	0.00113042756785831\\
6.336	0.00680179210235998\\
6.75	0.0291476383932157\\
6.966	0.0390684437699775\\
7.134	0.0453778465691421\\
7.204	0.0465787724126532\\
7.276	0.0462132696795106\\
7.379	0.0440085169378293\\
7.55	0.0386956127440659\\
7.757	0.0307383108911967\\
8.223	0.00984814217772723\\
8.305	0.00835406113259651\\
9.216	0.00287969892893969\\
9.295	0.000346926776547463\\
9.397	-0.00462027686181266\\
10	-0.0357808707504983\\
};
\addlegendentry{\structRealization}

\addplot [color=mycolor3,dotted,line width=2pt]
  table[row sep=crcr]{%
0	0\\
0.0310000000000006	0.000832721463400077\\
0.106999999999999	0.00489738723837263\\
0.539	0.0288109032213271\\
0.76	0.0393873946439633\\
0.952	0.0470720220700311\\
1.039	0.0494481715929922\\
1.093	0.0487856071252502\\
1.208	0.0455462686934212\\
1.374	0.0393890928579488\\
1.562	0.0308513003455069\\
1.788	0.0189925122762933\\
2.079	0.00366177130441336\\
2.159	0.00271572003195963\\
2.386	0.00200254202212946\\
3.096	0.000382800775160064\\
3.173	-0.00301495827056364\\
3.362	-0.0133104620594189\\
3.743	-0.0339783895801293\\
3.941	-0.0430605303962963\\
4.075	-0.0478772592634158\\
4.135	-0.0483307547584531\\
4.216	-0.0472492159457758\\
4.34	-0.0439637890437083\\
4.494	-0.0383300869863969\\
4.68	-0.0299396241035659\\
4.922	-0.0174043053728408\\
5.124	-0.00719267087388964\\
5.201	-0.00553827595436118\\
5.336	-0.00442418312579562\\
5.6	-0.00394328026872159\\
6.112	-0.00309668014037356\\
6.184	-0.00116561096117707\\
6.282	0.00317846283862977\\
6.465	0.0130521463933029\\
6.883	0.0358709974451479\\
7.069	0.0442608169620122\\
7.151	0.0467394153462841\\
7.222	0.0472885037561745\\
7.313	0.0463608437321668\\
7.434	0.043530538062786\\
7.58	0.0385482976598208\\
7.76	0.0308110354195268\\
8.002	0.0187766373692426\\
8.18	0.0102542947718121\\
8.265	0.00813221238779782\\
8.389	0.00673074424126163\\
8.597	0.00605949596801914\\
9.177	0.00484349710385779\\
9.259	0.00265274699581042\\
9.363	-0.00177912618562637\\
9.526	-0.0104097961619054\\
10	-0.0366443083343544\\
};
\addlegendentry{\structRealizationParam}

\end{axis}
\end{tikzpicture}%
	\end{subfigure}\hfill
	\begin{subfigure}[b]{.48\linewidth}
		\centering
		\input{ExampleDelay1-FOMvsROM-verification-simulation-absError-sin.tex}
	\end{subfigure}\\
	\begin{subfigure}[b]{.48\linewidth}
		\centering
		\input{ExampleDelay1-FOMvsROM-verification-simulation-triangle.tex}
	\end{subfigure}\hfill
	\begin{subfigure}[b]{.48\linewidth}
		\centering
		\input{ExampleDelay1-FOMvsROM-verification-simulation-absError-triangle.tex}
	\end{subfigure}\\
	\begin{subfigure}[b]{.48\linewidth}
		\centering
		\input{ExampleDelay1-FOMvsROM-verification-simulation-texp.tex}
		\subcaption{time domain simulation}
	\end{subfigure}\hfill
	\begin{subfigure}[b]{.48\linewidth}
		\centering
		\input{ExampleDelay1-FOMvsROM-verification-simulation-absError-texp.tex}
		\subcaption{absolute error}
	\end{subfigure}
	\caption{Comparison between the output of the original model and the outputs of the approximations. Top: $\inpVar_1$; middle: $\inpVar_2$; bottom: $\inpVar_3$.}
	\label{fig:timeSeries-ROMvsFOMsimulation}
\end{figure}

As already noted, all previous results have been obtained by exploiting the knowledge of the actual time delay which is equal to one in this case. However, in practical applications we cannot expect to have precise a priori knowledge of the time delay, but rather a rough estimate of it. Thus, in order to build the realization from data only, we modify the function $\hfunc_3$ as 
\begin{displaymath}
	\hfunc_3(s,\delay) = -\mathrm{e}^{-\delay s}
\end{displaymath} 
with free parameter $\delay$.
As discussed in \Cref{sec:parameterEstimation}, we can then use the test data to find an optimal delay time $\delay^\star$. A sampling of the least-squares error \eqref{eq:lsErrorFunctional} is provided in \Cref{fig:samplingLeastSquaresError} and reveals a distinct minimum at around the actual time delay $\tau=1$.

\begin{figure}
	\centering
	\input{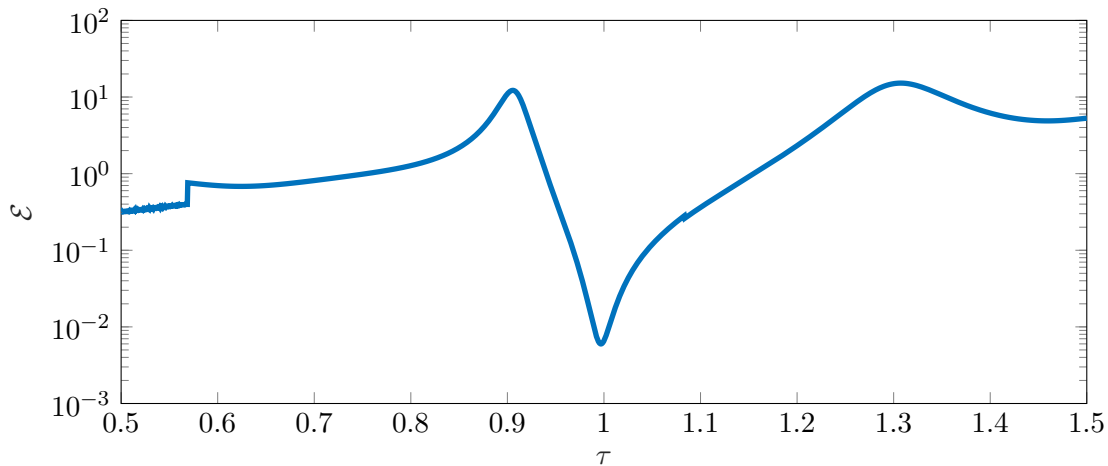}
	\caption{Sampling of the least-squares error \eqref{eq:lsErrorFunctional} over the delay time $\delay$.}
	\label{fig:samplingLeastSquaresError}
\end{figure}

The actual minimization of the cost function \eqref{eq:lsErrorFunctional} was performed using the MATLAB function \texttt{fmincon}.
As start value we used $\delay = \num{0.98}$, which was obtained from a rough sampling of the cost function.
The minimizer determined via \texttt{fmincon} is $\delay^\star = 0.996883$ and for this time delay the cost function attains a value of $\lsError(\delay^\star) = \num{5.99e-3}$. To simplify the numerical simulations, we use the rounded value $\delay^\star = 0.997$ for the following results. The transfer function of the realization constructed with the estimated parameter $\delay^\star$ and all transfer function estimates (i.e., the interpolation data and the test data), are depicted in \Cref{fig:lsETFE-FOMvsROM}. It is slightly different from the realization with the true parameter, but still approximates the original model well. This statement can be verified by the simulation results with the test inputs $\inpVar_1$, $\inpVar_2$, and $\inpVar_3$, which are presented in \Cref{fig:timeSeries-ROMvsFOMsimulation} and \Cref{tab:errorMeasurementsForValidationInputsParamOptim}.

\begin{table}
	\centering
	\caption{Error measurements for the validation inputs based on the estimated delay}
	\label{tab:errorMeasurementsForValidationInputsParamOptim}
	\begin{tabular}{llll}
		\toprule
		\textbf{input signal} & $\LtwoNorm{\inpVar}$ & $\frac{\LinftyNorm{\outVar-\outVarRed}}{\LtwoNorm{\inpVar}}$ & $\frac{\LtwoNorm{\outVar-\outVarRed}}{\LtwoNorm{\inpVar}}$\\
		\midrule
		$\inpVar_1$ & \num{2.18e+00} & \num{1.31e-03} & \num{1.77e-03}\\
		$\inpVar_2$ & \num{3.29e+00} & \num{2.43e-03} & \num{3.79e-03}\\
		$\inpVar_3$ & \num{3.96e-01} & \num{1.43e-02} & \num{1.04e-02}\\\bottomrule
	\end{tabular}
\end{table}

It is worth to note that there is no guarantee that the realization obtained from \Cref{alg:structuredRealization} is stable.
In order to investigate the stability of the obtained realizations, we consider the eigenvalues depicted in \Cref{fig:lsETFE-eigenvalues}, which are computed using the algorithm from \cite{WuM12}.
Indeed, the realization obtained from all transfer function estimates and the estimated delay $\delay^\star$ is unstable with one eigenvalue in the right half plane (cf.\  \Cref{fig:lsETFE-eigenvalues-ROMparam}).
This is not very surprising, since the eigenvalues of the original model (cf.\ \Cref{fig:lsETFE-eigenvalues-FOM}) are close to the imaginary axis, such that one can expect that a small perturbation results in an unstable model. Still, the realization constructed only from the interpolation data and with the true value for the delay $\delay$ is stable (see \Cref{fig:lsETFE-eigenvalues-ROM} for the eigenvalues with the largest real part). In contrast to stabilizing post-processing algorithms for rational realizations as offered in \cite{GosA16}, a stable--unstable decomposition of a \gls{DDAE} is not possible in general and thus stability must be enforced during the construction of the realization. This is currently under investigation and subject to further research.

\begin{figure}
	\begin{subfigure}[c]{1\linewidth}
		\centering
		\input{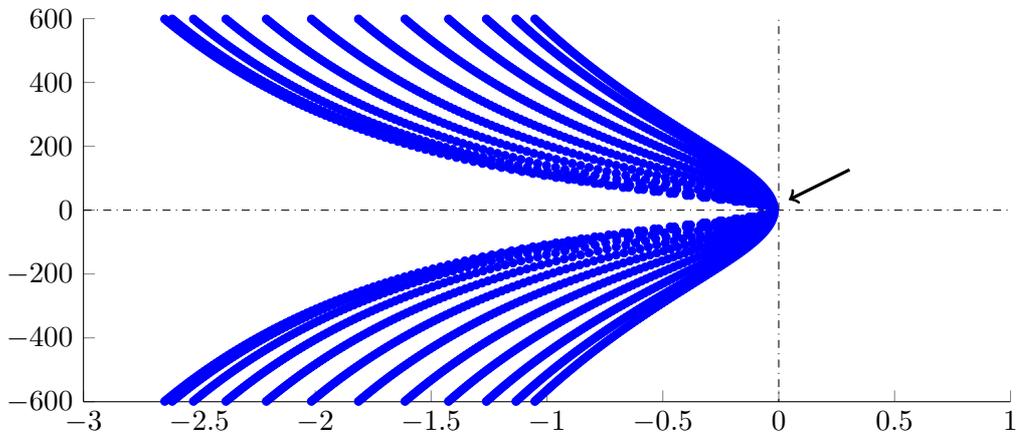}
		\subcaption{original model}
		\label{fig:lsETFE-eigenvalues-FOM}
	\end{subfigure}\\
	\begin{subfigure}[c]{\linewidth}
		\centering
		\input{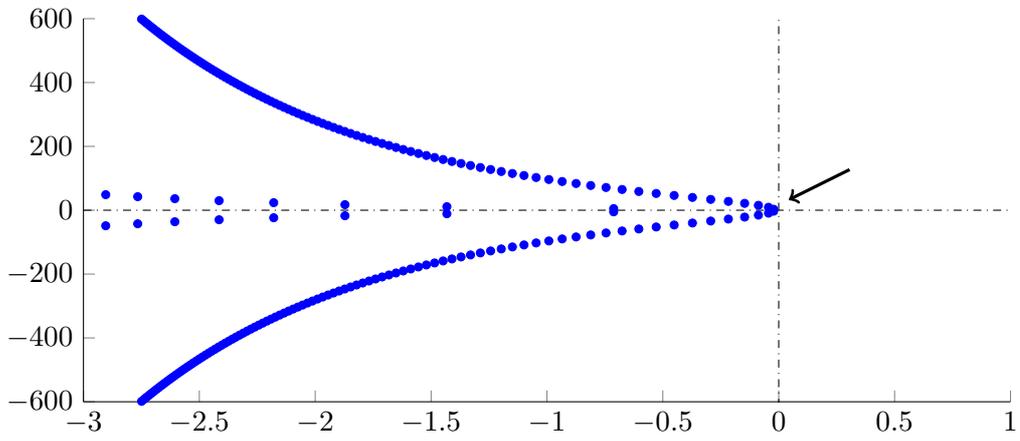}
		\subcaption{structured realization with true delay}
		\label{fig:lsETFE-eigenvalues-ROM}
	\end{subfigure}\\
	\begin{subfigure}[c]{\linewidth}
		\centering
		\input{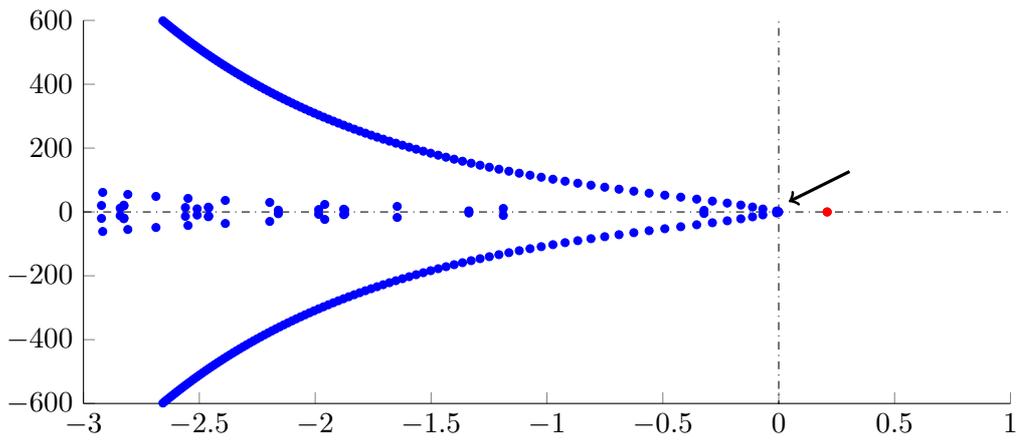}
		\subcaption{structured realization with estimated delay $\delay^\star$}
		\label{fig:lsETFE-eigenvalues-ROMparam}
	\end{subfigure}
	\caption{Eigenvalues of the realization with largest real part}
	\label{fig:lsETFE-eigenvalues}
\end{figure}

We conclude this case study with a remark about the choice of the interpolation frequencies.

\begin{remark}
	In our numerical simulations, we observed that including estimates of the transfer function at smaller frequencies tends to produce less unstable realizations in the sense that fewer eigenvalues are unstable and the real part of the unstable eigenvalues is smaller compared to a realization obtained from estimates of the transfer function at higher frequencies. As an example we refer to the realization obtained only from the interpolation data (with the true delay $\delay=1$) and the realization obtained from all transfer function estimates (with the estimated delay $\delay = \delay^*$), see \Cref{fig:lsETFE-eigenvalues} for the corresponding eigenvalue plots.
\end{remark}

\section{Summary}
We presented a framework for constructing structured realizations purely based on input and output trajectories in the time-domain.
The approach allows for a wide range of system structures such as second-order, time-delay, and viscoelastic systems.
The procedure can essentially be subdivided into two steps: First, we estimate transfer function data in the frequency domain based on given discrete-time data of the input and the output signal.
To this end, we apply the approach presented in \cite{PehGW17} to our setting and obtain transfer function estimates via solving a least-squares problem, cf.\ \Cref{alg:LSTFE}.
The second step is the construction of a structured realization based on the given transfer function estimates, see \Cref{alg:structuredRealization}.
For this purpose, we use the framework introduced in \cite{SchUBG18} to obtain a structured realization which interpolates the transfer function data generated in the first step.

In practical applications it may be advantageous to not completely fix the system structure a priori, but instead to allow for some flexibility via introducing free parameters.
Considering delay systems for example, the dynamics strongly depend on the value of the time delay, which can generally not be assumed to be known a priori.
For this purpose, we presented a simple procedure to choose these free parameters in an optimal way.
This is achieved by generating additional transfer function data and fit the parameters to these data via minimizing the data mismatch in a least squares sense.

We applied the complete framework to a delay example and observed that the obtained realization agrees well with the original model in terms of the transfer function and in terms of time-domain trajectories of the output using different input signals for validation.
This good agreement cannot only be observed when the true time delay is used, but also when we use an estimated time delay based on additional test data.

The presented results motivate for further research in this direction.
For instance, so far there is no guarantee that the realization is stable. 
Thus, it would be interesting to investigate how stability of the realization can be enforced during its construction.
This is still an open problem even for standard state-space realizations as considered in the Loewner framework.

\subsection*{Acknowledgments}
We thank Prof{.} Benjamin Peherstorfer for providing a MATLAB code for the numerical results obtained in \cite{PehGW17}.

\bibliographystyle{plain}
\bibliography{references}

\vfill
{\printglossary}

\end{document}